\documentclass[12pt,reqno]{amsart}
\usepackage{amsmath}

\usepackage{amsthm}
\usepackage{amssymb}
\usepackage{graphics}
\usepackage{latexsym}
\usepackage{comment}

\numberwithin{equation}{section}
\newtheorem{thm}{Theorem}[section]
\newtheorem{prop}[thm]{Proposition}
\newtheorem{lem}[thm]{Lemma}

\theoremstyle{remark}
\newtheorem{rem}{Remark}[section]
\newtheorem{defn}{Definition}

\newcommand{\BBB}{\mathbb}
\newcommand{\R}{{\BBB R}}
\newcommand{\Z}{{\BBB Z}}
\newcommand{\T}{{\BBB T}}

\newcommand{\M}{{\mathcal M}}
\newcommand{\N}{{\BBB N}}
\newcommand{\C}{{\BBB C}}
\newcommand{\LR}[1]{{\langle {#1} \rangle }}

\newcommand{\al}{\alpha}
\newcommand{\be}{\beta}
\newcommand{\ga}{\gamma}

\newcommand{\vp}{\varphi}

\newcommand{\e}{\varepsilon}
\newcommand{\ta}{\tau}
\newcommand{\x}{\xi}

\newcommand{\p}{\partial}
\newcommand{\la}{\lambda}

\newcommand{\de}{\delta}
\newcommand{\om}{\omega}

\newcommand{\equivalent}{\Leftrightarrow}

\newcommand{\I}{\infty}

\newcommand{\RE}{\operatorname{Re}}
\newcommand{\IM}{\operatorname{Im}}

\newcommand{\EQ}[1]{\begin{equation} \begin{split} #1
 \end{split} \end{equation}}
\newcommand{\EQS}[1]{\begin{align} #1 \end{align}}
\newcommand{\EQQS}[1]{\begin{align*} #1 \end{align*}}
\newcommand{\EQQ}[1]{\begin{equation*} \begin{split} #1
 \end{split} \end{equation*}}

\newcommand{\F}{\mathcal{F}}

\newcommand{\ha}{\widehat}

\newcommand{\ds}{\partial_{x}}%
\newcommand{\dt}{\partial_{t}}%
\newcommand{\zm}{P_{\neq0}}

\if0
\usepackage{lineno}
\newcommand*\patchAmsMathEnvironmentForLineno[1]{
  \expandafter\let\csname old#1\expandafter\endcsname\csname #1\endcsname
  \expandafter\let\csname oldend#1\expandafter\endcsname\csname end#1\endcsname
  \renewenvironment{#1}
     {\linenomath\csname old#1\endcsname}
     {\csname oldend#1\endcsname\endlinenomath}}
\newcommand*\patchBothAmsMathEnvironmentsForLineno[1]{
  \patchAmsMathEnvironmentForLineno{#1}
  \patchAmsMathEnvironmentForLineno{#1*}}
\AtBeginDocument{
\patchBothAmsMathEnvironmentsForLineno{equation}
\patchBothAmsMathEnvironmentsForLineno{align}
\patchBothAmsMathEnvironmentsForLineno{flalign}
\patchBothAmsMathEnvironmentsForLineno{alignat}
\patchBothAmsMathEnvironmentsForLineno{gather}
\patchBothAmsMathEnvironmentsForLineno{multline}
}
\linenumbers
\fi

\title[Schr\"odinger equations]
{Well-posedness and parabolic smoothing effect\\
 for higher order Schr\"odinger type equations with constant coefficients}

\author[T. Tanaka and K. Tsugawa]{Tomoyuki Tanaka and Kotaro Tsugawa}
\address[T. Tanaka]{Graduate School of Mathematics, Nagoya University,
Chikusa-ku, Nagoya, 464-8602, Japan}
\email[T. Tanaka]{d18003s@math.nagoya-u.ac.jp}
\address[K. Tsugawa]{Department of Mathematics, Faculty of Science and Engineering, Chuo University, Bunkyo-ku, Tokyo, 112-8551, Japan}
\email[K. Tsugawa]{tsugawa@math.chuo-u.ac.jp}

\begin{document}
\maketitle
\setcounter{page}{001}

\begin{abstract}
  In this paper, we consider the Cauchy problem of a class of higher order Schr\"odinger type equations with constant coefficients.
By employing the energy inequality, we show the $L^2$ well-posedness, the parabolic smoothing and a breakdown of the persistence of regularity.
We classify this class of equations into three types on the basis of their smoothing property.
\end{abstract}

%%%%%%%%%%%%%%%%%%%%%%%%%%%%%%%%%%%%%%%%%%%%%%%%%%
%%%%%%%%%%%%%%%%%%%%%%%%%%%%%%%%%%%%%%%%%%%%%%%%%%
%%%%%%%%%%%%%%%%%%%%%%%%%%%%%%%%%%%%%%%%%%%%%%%%%%
%%%%%%%%%%%%%%%%%%%%%%%%%%%%%%%%%%%%%%%%%%%%%%%%%%
%\section{Higher order Schr\"odinger equations}
\section{Introduction}
In this paper, we consider the Cauchy problem of the following:
\EQS{
  &D_t u(t,x)=D_x^{2m} u(t,x) +\sum_{j=1}^{2m}\big(a_j D_x^{2m-j}u(t,x)+b_j D_x^{2m-j}\bar{u}(t,x)\big),\label{2mLSE}\\
  &u(0,x)=\vp(x),\label{initial}
}
where $1\le m\in\N$, $\M = \R \, (\text{or} \, \T)$, $(t,x)\in (-\infty,\infty) \times\M$, $D_t=-i\dt$, $D_x=-i\ds$ and $i$ is the imaginary unit.
The constants $\{a_j\}, \{b_j\} \subset \C$ and the initial data $\vp(x): \M \to \C$
are given and $u(t,x): (-\infty,\infty) \times \M \to \C$ is unknown.
We are interested in the Cauchy problem of the following higher
order nonlinear Schr\"odinger type equations:
\EQ{\label{NLS}
  i\p_t u(t,x)-\p_x^{2m} u(t,x)= F(\p_x^{2m-1}u, \p_x^{2m-1}\overline{u},\p_x^{2m-2}u, \p_x^{2m-2}\overline{u}, \ldots u, \overline{u}),
}
with \eqref{initial}, where $F$ is a polynomial.
As important examples, this class of equations includes the nonlinear Schr\"odinger hierarchy and the derivative nonlinear Schr\"odinger hierarchy, which are integrable systems appearing in the soliton theory.
It is known that the Cauchy problem of \eqref{NLS} with \eqref{initial} is locally well-posed on $\R$ in weighted Sobolev spaces (of which functions are also sufficiently smooth).
Its proof is based on the Kato type smoothing estimate and the gauge transform \cite{Hayashi93, HO92}.
See Section 3 in \cite{KPV94} for this argument.
On the one hand, the well-posedness for \eqref{NLS} with \eqref{initial} on $\R$ without any weight or on $\T$ remains open.
We also refer to \cite{Grunrock10, KP16, KPV94, Schwarz84} for well-posedness results to higher order dispersive equations including the KdV hierarchy.
In \cite{Chihara02}, Chihara studied the well-posedness and the ill-posedness
of \eqref{NLS} for $m=1$ with \eqref{initial} on $\T$.
Recently, in \cite{TsugawaP}, the second author has studied a similar problem and
shown a non-existence result of solutions of \eqref{NLS} for some nonlinearity and $m=1$ with \eqref{initial}
on $\T$ by employing a smoothing for elliptic equations.
In other words, even when we restrict \eqref{NLS} to $m=1$, the well-posedness for \eqref{NLS} with \eqref{initial} on $\T$ is remarkably different from that on $\R$.
Therefore, the nonlinearity $F$ must have special structures (expected to include the case where \eqref{NLS} is a integrable system)
when the Cauchy problem of \eqref{NLS} with \eqref{initial} is (locally) well-posed on $\T$.
In proofs of \cite{Chihara02, TsugawaP}, the so called ``energy inequality'' of \eqref{2mLSE} with variable coefficients
$\{a_j(t,x)\}$ and $\{b_j(t,x)\}$ plays an important role.
Our plan is to extend this result to $m\ge 2$.
However, the energy inequality for higher $m$ is much complicated.
Therefore, we assume $\{a_j\}$ and $\{b_j\}$ are constants to make the problem simple
in the present paper and will study the variable coefficients case in the forthcoming paper.
$\la$ defined below is used to classify \eqref{2mLSE} into three types.
\begin{defn}\label{definition_lambda}
We write $\sum_{k=1}^0 c_k=0$ for any sequence $\{c_k\}$.
$\ga=\{\ga_j\}_{j=1}^{m-1}$ and $\la=\{\la_j\}_{j=1}^{2m-1}$ are defined as
\EQQS{
  &\ga_{j}=b_{2j}-\sum_{k=1}^{j-1}\bar{a}_{2(j-k)}\ga_k,\quad 1\le j\le m-1,\\
  &\begin{cases}
    \la_{2j}=2\IM a_{2j}
      -2\displaystyle{\sum_{k=1}^{j-1}}\IM\bar{b}_{2(j-k)}\ga_k, \quad 1\le j\le m-1,\\
    \la_{2j-1}=2\IM a_{2j-1}
      +2\displaystyle{\sum_{k=1}^{j-1}}\IM\bar{b}_{2(j-k)-1}\ga_k,\quad 1\le j\le m.
  \end{cases}
}
\end{defn}

Our main result is the following. We write $P^+ f(x):=\F^{-1}(\chi(\xi \ge 1)\F{f})(x)$ and $P^- f(x):=\F^{-1}(\chi(\xi \le -1)\F{f})(x)$, where $\F$ is the Fourier transform and $\chi$ is the definition function.
\begin{thm}\label{main_theorem}\hspace*{0em}\\
(Dispersive type, $L^2$ well-posedness)
Assume that $\la_j=0$ for $1\le j\le 2m-1$.
Then, for any $\vp\in L^2(\M)$, there exists a unique solution $u(t,x)$ of \eqref{2mLSE}--\eqref{initial} such that $u(t,x) \in C((-\infty,\infty); L^2(\M))$.\\
(Parabolic type)
Assume that there exists $j^*\in \N$ such that $\la_j=0$ for $1\le j < 2j^*$
and $\la_{2j^*}> 0$ (resp. $\la_{2j^*}< 0$).
(i) For any $\vp\in L^2(\M)$, there exist a unique solution $u(t,x)$ of \eqref{2mLSE}--\eqref{initial} on $[0,\infty)$ (resp. $(-\infty,0]$)
such that $u(t,x) \in C([0,\infty); L^2(\M))\cap C^\infty((0,\infty)\times\M)$ (resp. $C((-\infty,0]; L^2(\M))\cap C^\infty((-\infty,0)\times\M)$).
(ii) For any $\vp\in L^2(\M)\setminus C^\infty(\M)$ and $\de >0$, no solution $u$ of \eqref{2mLSE}--\eqref{initial} exists on $(-\de,0]$ (resp. $[0,\de)$)
such that $u(t,x) \in C((-\de,0]; L^2(\M))$ (resp. $C([0,\de); L^2(\M))$).\\
(Elliptic type)
Assume that there exists $j^*\in \N$ such that $\la_j=0$ for $1\le j < 2j^*-1$ and $\la_{2j^*-1} > 0$ (resp. $\la_{2j^*-1} < 0$).
(i) Let $\vp\in L^2(\M)$ satisfy $P^+ \vp \not\in H^{1/2}(\M)$. Then, for any $\delta>0$, there exist no solution $u(t,x)$ of \eqref{2mLSE}--\eqref{initial} on $[-\delta,0]$ (resp. $[0,\delta]$) satisfying $u \in C([-\delta,0];L^2(\M))$ (resp. $u \in C([0,\delta];L^2(\M))$).
Moreover, the same result as above holds even if we replace $P^+$, $[-\delta,0]$ and $[0,\delta]$ with $P^-$, $[0,\delta]$
and $[-\delta,0]$, respectively.
(ii) Let $\vp\in L^2(\M)\setminus C^\I(\M)$.
Then, for any $\delta>0$, there exist no solution $u(t,x)$ of \eqref{2mLSE}--\eqref{initial} on $[-\delta,\de]$ satisfying $u \in C([-\delta,\de];L^2(\M))$
\end{thm}

\begin{rem}
Put $v(t)= \LR{\p_x}^{-s} u(t)$. Then $v$ satisfies \eqref{2mLSE} if $u$ is the solution of \eqref{2mLSE} and $u(t) \in L^2(\M)  \equivalent v(t) \in H^s(\M)$.
Therefore, Theorem \ref{main_theorem} holds even if we replace $L^2(\M)$ with $H^s(\M)$ and $H^{1/2}(\M)$
with $H^{s+1/2}(\M)$ for any $s\in \R$.
\end{rem}
\begin{rem}
In ``Dispersive type,'' the persistence of regularity holds on both $(-\infty,0]$ and $[0,\infty)$.
In ``Parabolic type,'' the equations have the parabolic smoothing effect on either $(-\infty,0]$ or $[0,\infty)$, which means the persistence of regularity breaks down on either $[0,\infty)$ or $(-\infty,0]$.
Non-existence results in ``Parabolic type'' and ``Elliptic type'' is by the breakdown
of the persistence of regularity.
\end{rem}

\begin{rem}
  We give some examples of $\{a_j\}$ and $\{b_j\}$.
  \begin{itemize}
    \item When $m=1$, ``Parabolic type'' does not occur.
    In fact, the equation $D_t u=D_x^2 u+a_1 D_x u+a_2u+b_1D_x\bar{u}+b_2\bar{u}$ is ``Dispersive type'' if $\IM a_1=0$ and it is ``Elliptic type'' otherwise.
    \item Let $H(u)$ be a quadratic form defined by
    \EQQS{
      H(u)
      :=\frac{1}{2}\int|\ds^m u|^2
       +\sum_{j=1}^{2m}(c_ju\ds^{2m-j}u
         +d_j u\ds^{2m-j}\bar{u}
         +e_j\bar{u}\ds^{2m-j}\bar{u})dx
    }
    for given $\{c_j\},\{d_j\},\{e_j\}\subset\C$.
    Then, it is easy to check that $H(u)$ is the Hamiltonian of the equation \eqref{2mLSE} if and only if $\IM a_j=b_{2n-1}=\RE d_{2n-1}=\IM d_{2n}=0$ and $c_{2n}=\bar{e}_{2n}$ for $1\le j\le 2m-1$ and $1\le n\le m$.
    In particular, we can write $c_{2j}=(-1)^{m-j}\bar{b}_{2j}/2$,
    $d_{2j-1}=i(-1)^{m-j}a_{2j-1}$ and $d_{2j}=(-1)^{m-j}a_{2j}$ for $1\le j\le m$ (without loss of generality we can assume $c_{2j-1}=e_{2j-1}=0$ since $c_{2j-1}$- and $e_{2j-1}$-terms always vanish by the integration by parts).
    In this case, we see from Definition \ref{definition_lambda} and Remark \ref{rem_la} that $\la_j=0$ for $1\le j\le 2m-1$, which implies that Hamiltonian equations are ``Dispersive type.''
    \item By using the equation \eqref{2mLSE}, we have
    \EQQS{
      \frac{d}{dt}\|u(t)\|^2
      &=2\RE i\LR{D_t u,u}\\
      &=-2\sum_{j=1}^{2m}(\IM a_j)\LR{D_x^{2m-j}u,u}
       -2\sum_{n=1}^m \IM b_{2n}\LR{D_x^{2(m-n)}\bar{u},u}.
    }
    Therefore, when $\IM a_j=b_{2n}=0$ for $1\le j\le 2m-1$ and $1\le n\le m$, the solution of the equation \eqref{2mLSE} conserves the mass, i.e., $\|u(t)\|$.
    We see from the scaling argument that this condition is also necessary.
    In this case, the equation \eqref{2mLSE} is ``Dispersive type.''
    Indeed, it is easy to see $\ga_j=0$ for $1\le j\le m-1$ by Definition \ref{definition_lambda}.
    It then follows that $\la_k=2\IM a_k=0$ for $1\le k\le 2m-1$.
    \item When $m=2$, we have
    \EQQS{
      \la_1=2\IM a_1,\quad
      \la_2=2\IM a_2,\quad
      \la_3=2\IM a_3+2\IM\bar{b}_1b_2.
    }
    So, equations $D_tu=D_x^4u+iD_xu$ and $D_tu=D_x^4u+D_x^3\bar{u}-iD_x^2\bar{u}$ are ``Elliptic type.''
    On the other hand, $D_tu=D_x^4u+iD_xu+D_x^3\bar{u}-iD_x^2\bar{u}$ is ``Dispersive type'' although this equation does not have the Hamiltonian.
  \end{itemize}
\end{rem}

We recall several results for equations related to \eqref{2mLSE}.
There is a large literature on the well-posedness for the Cauchy problem of Schr\"odinger type equations, especially $m=1$.
In \cite{Mizohata81}, Mizohata showed that if the Cauchy problem
\EQQS{
  %\label{MT1}
  &\dt u=\sum_{j=1}^n(i\p_j^2 + c_j(x)\p_j) u+f(t,x),\quad (t,x)\in\R\times\R^n,\\
  %\label{MT2}
  &u(0,x)=\vp(x)
}
is $L^2$ well-posed, then the condition
\EQQS{
  \sup_{(t,\om,x)\in\R\times S^{n-1}\times \R^n}\Biggr|\IM\int_0^t \sum_{j=1}^n c_j(x+s\om)\om_j ds\Biggr|<\I
}
holds.
In particular, this condition is also sufficient for the $L^2$ well-posedness when $n=1$.
%However, when $n\ge2$, some stronger sufficient conditions were proposed by several authors.
See \cite{Akhunov14, Chihara02, Doi96, Ichinose84, Kajitani98, Mizohata85, Takeuchi80, Tarama93} (and references therein) for related results.
For $m=2$, in \cite{Mizuhara06}, Mizuhara studied $L^2$ well-posedness for the Cauchy problem:
\EQS{
  \label{Miz1}
  &(D_t-D_x^4-c_1(x)D_x^3-c_2(x)D_x^2-c_3(x) D_x-c_4(x))u=f(t,x)\\
  \label{Miz2}
  &u(0,x)=\vp(x),
}
where $(t,x)\in\R\times\M$.
To be precise, he also studied another equation of the KdV type.
When $\M=\T$, he deduced the necessary and sufficient conditions for the $L^2$ well-posedness for \eqref{Miz1}--\eqref{Miz2}.
On the other hand, when $\M=\R$, he showed some conditions for the $L^2$ well-posedness.
Indeed, his sufficient condition for the $L^2$ well-posedness is also necessary under the additional assumption.
In \cite{Tarama11}, Tarama removed Mizuhara's additional assumption, so he obtained the necessary and sufficient conditions for the $L^2$ well-posedness for \eqref{Miz1}--\eqref{Miz2} on $\R$.

Since the coefficients are constants, by the Fourier transform, \eqref{2mLSE} can be rewritten into the following:
\EQ{\label{2mLSEFourier}
D_t \ha{u}(t,\x)=\x^{2m} \ha{u}(t,\x) +\sum_{j=1}^{2m}\big(a_j \x^{2m-j}\ha{u}(t,\x)+b_j \x^{2m-j}\overline{\ha{u}}(t,-\x)\big).
}
Here, we fix $\x\in \R$ (or $\Z$) and put
\EQQ{
  U_\x(t)=\left(
    \begin{array}{cc}
      \ha{u}(t,\x)\\
      \overline{\ha{u}}(t,-\x)
    \end{array}
  \right),
\ \ \
  X_0 = \left(
    \begin{array}{cc}
      1 & 0 \\
      0 & -1
    \end{array}
  \right),
\ \ \
  X_j = \left(
    \begin{array}{cc}
      a_j & b_j \\
      (-1)^{j+1}\overline{b_j} & (-1)^{j+1}\overline{a_j}
    \end{array}
  \right),
}
for $1\le j \le 2m$.
Then, by \eqref{2mLSEFourier} with \eqref{initial}, it follows that
\EQ{\label{ODE}
D_t U_\x(t) =\sum_{j=0}^{2m} \x^{2m-j} X_j U_\x(t), \ \ \ U_\x(0)={}^t(\ha{\vp}(\x),\overline{\ha{\vp}}(-\x)),
}
which is a system of linear ordinary differential equations. We can easily obtain the unique solution
\EQ{\label{explicit_formula}
U_\x(t)=U_\x(0) \exp it \sum_{j=0}^{2m} \x^{2m-j} X_j
}
on $t\in (-\infty,\infty)$ for each $\x \in \R$ (or $\Z$).
Therefore, our interest in Theorem \ref{main_theorem} is essentially on the regularity of the solution.
Here, note that $X_jX_k=X_kX_j$ holds for any $0 \le j,k \le 2m$ if and only if
$b_j=0$ holds for any $1 \le j\le 2m$.
If we assume this assumption, \eqref{ODE} is not a system but a single ordinary differential equation and
\EQ{
\ha{u}(t,\x)=\ha{\vp}(\x) \exp it\Big(\x^{2m}+\sum_{j=1}^{2m}\x^{2m-j}a_j\Big)
}
for each $\xi\in \R$ (or $\Z$).
Since $\ga_j=0$ and $\la_j=2\mathrm{Im} \, a_j$, it follows that
\EQQ{
|\ha{u}(t,\x)|=|\ha{\vp}(\x)| \prod_{j=1}^{2m} \exp \frac{-t\x^{2m-j}\la_j}{2},
}
by which we obtain Theorem \ref{main_theorem} easily.
On the other hand, it seems difficult to obtain Theorem \ref{main_theorem} by \eqref{explicit_formula} for general $\{b_j\}$
since $X_jX_k \neq X_kX_j$ for some $j,k$.
To avoid this difficulty, we employ the energy estimate.
In particular, we modify the energy, adding correction terms so as to cancel out derivative losses. See e.g. \cite{HITW15, Kwon08}  for this argument.
However, some of derivative losses cannot be eliminated by correction terms, and they may essentially affect the well-posedness of the Cauchy problem of \eqref{2mLSE}--\eqref{initial} as stated in Theorem \ref{main_theorem}.
Propositions \ref{prop_energy_est_0} and \ref{prop_energy_est_twist_+} are main estimates in this paper.
The first term of the left-hand side of \eqref{energy_ineq_0} is the main part of the energy.
The second term is the correction term.
For ``Dispersive type,'' the third and the fourth terms vanish.
Thus, we easily obtain the $L^2$ a priori estimate.
For ``Parabolic type,'' the third term includes $\la_{2j^*}\| |\p_x|^{m-j^*}u\|^2$.
The parabolic smoothing is caused by the term.
For ``Elliptic type,'' the fourth term includes $\la_{2j^*-1}\LR{D_x^{2(m-j^*)+1}u,u}$.
We want to show the parabolic smoothing by making use of the term.
However, the sign of the term is not definite. That is unfavorable in our argument.
Therefore, we compute the energy inequalities of $P^+u$ and $P^-u$ instead of $u$ and obtain
Lemma \ref{prop_energy_est_trc}.
Note that the sign of all terms except the correction terms in \eqref{energy_ineq_+} and \eqref{energy_ineq_-}
are definite.
Though \eqref{energy_ineq_+} is the energy inequality for $\|P^+u\|$, it includes $\la_j^-\| |\p_x|^{m-j/2}P^{-}u\|^2$.
This is because \eqref{2mLSE} is essentially coupled system of $P^+u$ and $P^-u$ as \eqref{2mLSEFourier}.
The term $\la_j^-\| |\p_x|^{m-j/2}P^{-}u\|^2$ cannot be estimated by $\|u\|$.
This is the main difficulty in the proof of ``Elliptic type'' in Theorem \ref{main_theorem}.
The key idea is to eliminate these terms in two steps where $1\le j\le 2j^*+1$ and $2(j^*+1)\le j\le 2m-1$.
First we analyse a property of $\{\la_j^-\}$ so that $\la_{2k}=0$ for $1\le k\le j^*-1$ implies that the first $2j^*+1$ of $\{\la_j^-\}$ vanish (see Lemma \ref{lem_la}).
In order to cancel out the rest of unfavorable terms $\la_j^-\| |\p_x|^{m-j/2}P^{-}u\|^2$ for $2(j^*+1)\le j\le 2m-1$, we use an additional correction term $F_k^-$ and obtain \eqref{eq_twist_+} (see also \eqref{eq_twist_-}).
Here, $F_k^-$ originates from the energy inequality for $\||\ds|^{-(k+2)/2} P^- u\|$, and $F_k^-$ does not yield a bad effect thanks to the first step.

The rest of this paper is organized as follows.
In Section 2, we state main estimates which are energy estimates for $u$ and $P^\pm u$, and give proofs of them.
In Section 3, we show Theorem \ref{main_theorem}.
In particular, we show a smoothing for ``Elliptic type" (Proposition \ref{prop_smoothing}) from the energy estimate for $P^\pm u$.

Here, we set some notation.
Let $\LR{\cdot,\cdot}:=\LR{\cdot,\cdot}_{L^{2}}$ and $\|\cdot\|:=\|\cdot\|_{L^2}$.
We also use the same symbol for $\LR{\cdot}:=(1+|\cdot|^{2})^{1/2}$.
$P_0$ and $P_{\neq 0}$ are defined by $P_0 f(x):=\F^{-1}(\chi(|\xi| < 1)\F{f})(x)$ and $P_{\neq 0} f(x):=\F^{-1}(\chi(|\xi| \ge 1)\F{f})(x)$.
We define the Riesz and Bessel potentials by $|\p_x|^s f:=\F^{-1}(|\x|^s\F{f})(x)$ and $\LR{\p_x}^s f:=\F^{-1}(\LR{\x}^s\F{f})(x)$.

%%%%%%%%%%%%%%%%%%%%%%%%%%%%%%%%%%%%%%%%%%%%%%%%%%%%%%%%%%%%%%%%%%%%%%%%%%%%%%%
%%%%%%%%%%%%%%%%%%%%%%%%%%%%%%%%%%%%%%%%%%%%%%%%%%%%%%%%%%%%%%%%%%%%%%%%%%%%%%%
%%%%%%%%%%%%%%%%%%%%%%%%%%%%%%%%%%%%%%%%%%%%%%%%%%%%%%%%%%%%%%%%%%%%%%%%%%%%%%%
%%%%%%%%%%%%%%%%%%%%%%%%%%%%%%%%%%%%%%%%%%%%%%%%%%%%%%%%%%%%%%%%%%%%%%%%%%%%%%%
%%%%%%%%%%%%%%%%%%%%%%%%%%%%%%%%%%%%%%%%%%%%%%%%%%%%%%%%%%%%%%%%%%%%%%%%%%%%%%%
%%%%%%%%%%%%%%%%%%%%%%%%%%%%%%%%%%%%%%%%%%%%%%%%%%%%%%%%%%%%%%%%%%%%%%%%%%%%%%%

\section{the energy estimates}
Our purpose in this section is to show Propositions \ref{prop_energy_est_0} and \ref{prop_energy_est_twist_+}.
Proposition \ref{prop_energy_est_0} below is used to show ``Dispersive type'' and ``Parabolic type'' in Theorem \ref{main_theorem}.
\begin{prop}\label{prop_energy_est_0}
  Let $u$ satisfy \eqref{2mLSE}.
  Then, there exists $C=C(\{a_j\},\{b_j\})>0$ such that
  \EQ{\label{energy_ineq_0}
  &\Biggr|\frac{d}{dt}\Biggr(\|u\|^2+\sum_{j=1}^{m-1}\RE\ga_{j}\LR{D_x^{-2j}\zm \bar{u},\zm u}\Biggr)\\
  &+\sum_{j=1}^{m-1}\la_{2j}\||\ds|^{m-j} u\|^2
   +\sum_{j=1}^{m}\la_{2j-1}\LR{D_x^{2(m-j)+1} u, u}\Biggr|
  \le C\|u\|^2.
  }
\end{prop}
\begin{defn}\label{def2}
$\al=\{\al_j\}_{j=1}^{2m-1},\la^+=\{\la_j^+\}_{j=1}^{2m-1},\la^-=\{\la_j^-\}_{j=1}^{2m-1}$ are defined as
  \EQQS{
    &\al_{j}=b_{j}-\displaystyle{\frac{1}{2}\sum_{k=1}^{j-1}} (1+(-1)^{j-k}) \bar{a}_{j-k}\al_k,\quad 1\le j\le 2m-1,\\
    &\la_j^+=2\IM a_j+\sum_{k=1}^{j-1}(-1)^{j-k+1}\IM\bar{b}_{j-k}\al_k,\quad 1\le j\le 2m-1,\\
    &\la_j^-=-\sum_{k=1}^{j-1}\IM \bar{b}_{j-k}\al_k,\quad 1\le j\le 2m-1.
  }
  Let $1\le j^*\le m-2$.
  Assume that $\la_{2j^*-1}^+\neq0$.
  $\be^+=\{\be_k^+\}_{k=1}^{2(m-j^*-1)}$ and $\be^-=\{\be_k^-\}_{k=1}^{2(m-j^*-1)}$ are inductively defined as
  \EQQS{
    &\la_{2j^*+k+1}^-
    =\sum_{j=1}^k(-1)^{k-j}\la_{2j^*+k-j-1}^+\be_j^+,
    \quad 1\le k\le 2(m-j^*-1),\\
    &\la_{2j^*+k+1}^-
    =\sum_{j=1}^k (-1)^{k}\la_{2j^*+k-j-1}^+\be_j^-,
    \quad 1\le k\le 2(m-j^*-1).
  }
  Note that $\la_{2j^*+k-j-1}^+=\la_{2j^*-1}^+\neq0$ when $j=k$.
  So, $\be_{k}^+$ and $\be_k^-$ are well-defined.
\end{defn}
\begin{rem}\label{rem2.1}
  It is easy to see that $\ga_j=\al_{2j}$ for $1\le j\le m-1$.
  Thus, we have
  \EQQ{
     \la_{2j}=\la_{2j}^+ + \la_{2j}^-,\quad
     \la_{2k-1}=\la_{2k-1}^+-\la_{2k-1}^-
  }
  for $1\le j\le m-1$ and $1\le k\le m$.
\end{rem}
Proposition \ref{prop_energy_est_twist_+} below is used to show ``Elliptic type'' in Theorem \ref{main_theorem}.
\begin{prop}\label{prop_energy_est_twist_+}
  Let $u$ satisfy \eqref{2mLSE}.
  Assume that there exists $j^* \in \N$ such that $\la_j=0$ for $1 \le j\le 2(j^*-1)$ and $\la_{2j^*-1} \neq 0$.
  Put
  \EQQS{
     &F_k^-(u)=\||\p_x|^{-(k+2)/2} P^- u\|^2
       +\sum_{j=1}^{2m-1}\RE\al_j\LR{D_x^{-j}|\ds|^{-k-2} \overline{P^+ u}, P^- u},\\
     &F_k^+(u)=\||\p_x|^{-(k+2)/2} P^+ u\|^2
       +\sum_{j=1}^{2m-1}\RE\al_j\LR{D_x^{-j}|\ds|^{-k-2} \overline{P^- u}, P^+ u}.
  }
  Then, there exists $C=C(\{a_j\},\{b_j\})>0$ such that
  \EQ{\label{eq_twist_+}
    &\Biggr|\frac{d}{dt}\Biggr(\|P^+ u\|^2
    +\sum_{j=1}^{2m-1}\RE\al_j\LR{D_x^{-j} \overline{P^- u}, P^+ u}+\sum_{k=1}^{2(m-j^*-1)} \be_k^+ F_k^-(u) \Biggr)\\
    &+ \la_{2j^*-1}^+\||\ds|^{m-j^*+1/2}P^+ u\|^2 \Biggr| \le C \|u\|^2+ C\||\ds|^{m-j^*}P^+ u\|^2,
  }
and
\EQ{\label{eq_twist_-}
    &\Biggr|\frac{d}{dt}\Biggr(\|P^- u\|^2
    +\sum_{j=1}^{2m-1}\RE\al_j\LR{D_x^{-j} \overline{P^+ u}, P^- u}+\sum_{k=1}^{2(m-j^*-1)} \be_k^- F_k^+(u) \Biggr)\\
    &- \la_{2j^*-1}^+\||\ds|^{m-j^*+1/2}P^- u\|^2 \Biggr| \le C \|u\|^2+ C\||\ds|^{m-j^*}P^- u\|^2.
  }
\end{prop}
To prove Propositions \ref{prop_energy_est_0} and \ref{prop_energy_est_twist_+}, we use the following lemma.
\begin{lem}\label{prop_energy_est_trc}
  Let $u$ satisfy \eqref{2mLSE}.
  Then, there exists $C=C(\{a_j\},\{b_j\})>0$ such that
  \EQ{\label{energy_ineq_+}
    &\Biggr|\frac{d}{dt}\Biggr(\|P^+ u\|^2
    +\sum_{j=1}^{2m-1}\RE\al_j\LR{D_x^{-j} \overline{P^- u}, P^+ u}\Biggr)\\
    &+\sum_{j=1}^{2m-1}(
    \la_j^+\||\ds|^{m-j/2}P^+ u\|^2
    +\la_j^-\||\ds|^{m-j/2}P^- u\|^2)\Biggr| \le C\|u\|^2
  }
  and
  \EQ{\label{energy_ineq_-}
    &\Biggr|\frac{d}{dt}\Biggr(\|P^- u\|^2
    +\sum_{j=1}^{2m-1}\RE\al_j\LR{D_x^{-j} \overline{P^+ u}, P^- u}\Biggr)\\
    &+\sum_{j=1}^{2m-1}(-1)^j
    (\la_j^+\||\ds|^{m-j/2}P^- u\|^2
    +\la_j^-\||\ds|^{m-j/2}P^+ u\|^2)\Biggr| \le C\|u\|^2.
  }
\end{lem}
\begin{proof}[Proof of Lemma \ref{prop_energy_est_trc}]
  First, we show \eqref{energy_ineq_+}.
  For simplicity, we set $v^+:=P^+ u$ and $v^-:=P^- u$.
  Note that $P^+ \bar{u}=\overline{P^- u}=\overline{v^-}$ and
  $P^- \bar{u}=\overline{P^+ u}=\overline{v^+}$.
  Then, $v^+$ and $v^-$ satisfy
  \EQS{\label{eq_v1}
  D_t v^+
  =D_x^{2m}v^+
   +\sum_{k=1}^{2m}(a_kD_x^{2m-k}v^+ + b_k D_x^{2m-k}\overline{v^-})
  }
  and
  \EQS{\label{eq_v2}
  D_t \overline{v^-}
  =-D_x^{2m}\overline{v^-}
   -\sum_{k=1}^{2m}(-1)^k(\bar{a}_kD_x^{2m-k}\overline{v^-}+\bar{b}_k D_x^{2m-k}v^+).
  }
  By \eqref{eq_v1}, we have
  \EQQ{
    \frac{d}{dt}\|v^+\|^2&=2\RE \LR{\dt v^+,v^+}=-2\IM \LR{D_t v^+,v^+}\\
    &=-2\sum_{j=1}^{2m}(\IM a_j\LR{D_x^{2m-j} v^+, v^+}
    +\IM b_j\LR{D_x^{2m-j} \overline{v^-}, v^+})\\
    &=-2\sum_{j=1}^{2m}(\IM a_j\||\p_x|^{m-j/2}P^+ u\|^2
    +\IM b_j\LR{D_x^{2m-j} \overline{v^-}, v^+}).
  }
Here, we consider the time derivative of correction terms to cancel out the second term.
  Fix $1\le j\le 2m-1$.
  We see from \eqref{eq_v1} and \eqref{eq_v2} that
  \EQQS{
    &\frac{d}{dt}\RE\al_j\LR{D_x^{-j} \overline{P^- u}, v^+}
       =-\IM\al_j\LR{D_x^{-j} D_t \overline{v^-}, v^+}+\IM\al_j\LR{D_x^{-j} \overline{v^-}, D_t v^+}\\
    &=\IM\al_j\LR{D_x^{-j}(D_x^{2m}\overline{v^-}),v^+}+\IM\al_j\LR{D_x^{-j}\overline{v^-}, D_x^{2m}v^+}\\
    &\quad+\sum_{k=1}^{2m}((-1)^k\IM\al_j \bar{a}_k \LR{D_x^{2m-k-j} \overline{v^-}, v^+}
     +(-1)^k\IM\al_j \bar{b}_k \LR{D_x^{2m-k-j} v^+, v^+}\\
    &\quad+\IM  \al_j \bar{a}_k\LR{D_x^{2m-k-j} \overline{v^-}, v^+}
     +\IM\al_j \bar{b}_k \LR{D_x^{2m-k-j} \overline{v^-}, \overline{v^-}})\\
    &=:A_1^j+B_1^j+\sum_{k=1}^{2m}(A_{2,k}^j+A_{3,k}^j+B_{2,k}^j+B_{3,k}^j).
  }
  Observe that
  \EQQS{
    &A_1^j+B_1^j
    =2\IM\al_j\LR{D_x^{2m-j} \overline{v^-}, v^+},\\
    &A_{2,k}^j+B_{2,k}^j
    =(1+(-1)^k)\IM\al_j \bar{a}_k\LR{D_x^{2m-k-j}\overline{v^-}, v^+},\\
    &A_{3,k}^j=(-1)^k\IM\al_j \bar{b}_k \||\p_x|^{m-(k+j)/2} P^+ u\|^2,\\
    &B_{3,k}^j=\IM\al_j \bar{b}_k \||\p_x|^{m-(k+j)/2} P^- u\|^2.
  }
  We collect coefficients of derivative losses with rearranging the summation order.
  Note that for any sequence $\{c_{j,k}\}$, it holds that
  \EQS{\label{eq4.1}
  \begin{aligned}
    \sum_{j=1}^{p}\sum_{k=1}^{p-j}c_{j,k}
    =\sum_{j=1}^{p-1}\sum_{k=0}^{p-1-j}c_{j,k+1}
    =\sum_{j=1}^{p-1}\sum_{k=1}^{j}c_{k,j-k+1}.
  \end{aligned}
  }
  It is easy to see that
  \EQQS{
  \Biggr|\sum_{j=1}^{2m}\sum_{k=2m-j}^{2m} (A_{2,k}^j+A_{3,k}^j+B_{2,k}^j+B_{3,k}^j)\Biggr|
  \lesssim\|u\|^2.
  }
  Then, by \eqref{eq4.1}, we have
  \EQQS{
    \sum_{j=1}^{2m-1}\sum_{k=1}^{2m-1-j}(A_{2,k}^j+B_{2,k}^j)
    &=\sum_{j=1}^{2(m-1)}\sum_{k=1}^{j}(A_{2,j-k+1}^k+B_{2,j-k+1}^k)\\
    &=\sum_{j=1}^{2(m-1)}\sum_{k=1}^{j}
     (1+(-1)^{j-k+1})\IM\al_k \bar{a}_{j-k+1}\LR{D_x^{2m-1-j}\overline{v^-}, v^+}.
  }
  Similarly, we obtain
  \EQQS{
    &\sum_{j=1}^{2m-1}\sum_{k=1}^{2m-1-j}A_{3,k}^j
    =\sum_{j=1}^{2(m-1)}\sum_{k=1}^{j}(-1)^{j-k+1}
      \IM\al_k \bar{b}_{j-k+1} \||\p_x|^{m-(j+1)/2} P^+ u\|^2,\\
    &\sum_{j=1}^{2m-1}\sum_{k=1}^{2m-1-j}B_{3,k}^j
    =\sum_{j=1}^{2(m-1)}\sum_{k=1}^{j}
      \IM\al_k \bar{b}_{j-k+1} \||\p_x|^{m-(j+1)/2} P^- u\|^2.
  }
  This concludes the proof of \eqref{energy_ineq_+}.
  For the proof of \eqref{energy_ineq_-}, we set $v^+:=P^- u$ and $v^-:=P^+ u$.
Then, they satisfy \eqref{eq_v1} and \eqref{eq_v2}. Therefore, the exactly same proof works.
\end{proof}

Now we are ready to prove Proposition \ref{prop_energy_est_0}.
Though we can prove it directly without using Lemma \ref{prop_energy_est_trc}, we give the proof of it by the lemma.
\begin{proof}[Proof of Proposition \ref{prop_energy_est_0}]
  Note that $\LR{P^+ f,P^- g}=\LR{P^- f,P^+ g}=0$ for any functions $f,g$.
  This implies that
  $\LR{\zm \bar{f},\zm g}=\LR{\overline{P^- f},P^+ g}+\LR{\overline{P^+ f},P^- g}$.
  So, collecting \eqref{energy_ineq_+} and \eqref{energy_ineq_-}, we obtain
  \EQQS{
    &\Biggr|\frac{d}{dt}\Biggr(\|\zm u\|^2+\sum_{j=1}^{m-1}\RE\al_{2j}\LR{D_x^{-2j}\zm \bar{u},\zm u}\Biggr)\\
    &+\sum_{j=1}^{m-1}\la_{2j}\||\ds|^{m-j}\zm u\|^2
     +\sum_{j=1}^{m}\la_{2j-1}\LR{D_x^{2(m-j)+1}\zm u,\zm u}\Biggr|
    \le C\|u\|^2.
  }
  We also note that $\ga_k=\al_{2k}$.
  Finally, it is easy to see that
  \EQQS{
    \Biggr|\frac{d}{dt}\|P_0 u\|^2+\sum_{j=1}^{m-1}\la_{2j}\||\ds|^{m-j}P_0 u\|^2+\sum_{j=1}^{m}\la_{2j-1}\LR{D_x^{2(m-j)+1}P_0 u,P_0 u}\Biggr|\le C\|u\|^2.
  }
  Therefore, we have \eqref{energy_ineq_0}.
\end{proof}

The terms $\la_j^-\||\ds|^{m-j/2}P^- u\|^2$ (resp. $\la_j^-\||\ds|^{m-j/2}P^+ u\|^2$) in \eqref{energy_ineq_+} (resp. \eqref{energy_ineq_-}) with $1\le j\le 2j^*-1$ in Lemma \ref{prop_energy_est_trc} are unfavorable in our argument to prove Proposition \ref{prop_energy_est_twist_+}.
Indeed, Proposition \ref{prop_energy_est_twist_+} is used to show ``Elliptic type'' in Theorem \ref{main_theorem} when $\la_{2j^*-1}\neq0$ under the assumption $\la_j=0$ for $1\le j\le 2(j^*-1)$.
So, we analyse the coefficients $\la^-$ below in order to ensure the condition $\la_j=0$ for $1\le j\le 2(j^*-1)$ implies $\la_j^+=\la_n^-=0$ for $1\le j\le 2(j^*-1)$ and $1\le n\le 2j^*-1$.
%In fact, we can get more properties about $\la_j^-$ than stated here (see Lemma \ref{lem_la}).
\begin{lem}\label{rem_recur}
  It holds that
  \EQQS{
    \la_{j+1}^-
    &=-\frac{1}{2}\sum_{l=1}^{j-1}(1+(-1)^l)
      (\RE a_l)\la_{j+1-l}^-\\
    &\quad+\frac{1}{2}\sum_{l=1}^{j-1}\sum_{k=1}^{j-l}
      (1+(-1)^{l})(\IM a_l)\RE\bar{b}_{j-l-k+1}\al_k
  }
  for $1\le j\le 2(m-1)$.
\end{lem}
\begin{proof}
  By the definitions of $\la_j^-$ and $\al_k$, we have
  \EQQS{
    \la_{j+1}^-
    =-\sum_{l=1}^{j} \IM b_l \bar{b}_{j-l+1}
     +\frac{1}{2}\sum_{l=1}^{j}\sum_{k=1}^{l-1}
      (1+(-1)^{l-k}) \IM \bar{b}_{j-l+1} \bar{a}_{l-k}\al_k
    =:A+B.
  }
  It is easy to see that $A=0$.
  Observe that
  \EQS{\label{eq1.1}
    \sum_{l=1}^{p}\sum_{k=1}^{l-1}c_{l-k}d_l e_k
    =\sum_{l=1}^{p-1}\sum_{k=1}^{p-l}c_l d_{l+k} e_k
  }
  for any sequences $\{c_j\},\{d_j\}$ and $\{e_j\}$.
  This implies that
  \EQQS{
    B
    &=\frac{1}{2}\sum_{l=1}^{j-1}\sum_{k=1}^{j-l}
     (1+(-1)^l)\IM\bar{b}_{j-l-k+1}\bar{a}_l\al_k\\
    &=\frac{1}{2}\sum_{l=1}^{j-1}\sum_{k=1}^{j-l}
     (1+(-1)^l)((\RE a_l)\IM\bar{b}_{j-l-k+1}\al_k
     -(\IM a_l)\RE\bar{b}_{j-l-k+1}\al_k).
  }
  Here we used the fact that $\IM cd=(\RE c)\IM d+(\IM c)\RE d$ for any $c,d\in\C$.
  This completes the proof.
\end{proof}
\begin{lem}\label{lem_la}
  Assume that there exists $j^*\in\N$ such that $\la_{2j}=0$ for $1\le j\le j^*$.
  Then, it holds that $\IM a_{2j}=\la_{2j}^+=0$ for $1\le j\le j^*$ and $\la_{j}^-=0$ for $1\le j\le 2j^*+3$.
\end{lem}
\begin{proof}
  First note that $\la_1^-=\la_2^-=\la_3^-=0$ even without the hypothesis.
  Indeed, it is clear that $\la_1^-=0$.
  We also have $\la_2^-=-\IM\bar{b}_1\al_1=0$ and $\la_3^-=-\IM\bar{b}_2\al_1-\IM\bar{b}_1\al_2=0$ since $\al_1=b_1$ and  $\al_2=b_2$.
  Assume that there exists $j^*\in\N$ such that $\la_{2j}=0$ for $1\le j\le j^*$.
  The rest of proof proceeds by the induction on $j$.
  We prove the following claim: it holds that
  $$\IM a_{2j}=\la_{2j}^+=\la_{2j+2}^-=\la_{2j+3}^-=\sum_{k=1}^{j-1}\IM\bar{b}_{2(j-k)}\al_{2k}=0$$
  for $1\le j\le j^*$.
  It is easy to see that the claim above with $j=1$ follows.
  Indeed, by the definition of $\al_j$, we obtain
   $\al_3=b_3-\bar{a}_2b_1$ and $\al_4=b_4-\bar{a}_2b_2$, which implies that
  \EQQS{
    \la_4^-
    &=-\IM\bar{b}_3b_1-\IM|b_2|^2-\IM\bar{b}_1b_3+\IM|b_1|^2\bar{a}_2=0,\\
    \la_5^-
    &=-\IM\bar{b}_4b_1-\IM\bar{b}_3b_2-\IM\bar{b}_2(b_3-\bar{a}_2b_1)
    -\IM\bar{b}_1(b_4-\bar{a}_2b_2)=0
  }
  since $\la_2=2\IM a_2=0$. We also have $\la_2^+=\IM\bar{b}_2\al_2=0$ easily.
  Next, we assume that the claim above holds for $j(\le j^*-1)$.
  By the hypothesis, it holds that $\la_{2j+2}^-=\la_{2j+3}^-=0$.
  Thus, by Remark \ref{rem2.1}, we have $\la_{2j+2}^+=0$.
  We claim that $M:=\sum_{l=1}^{j}\IM\bar{b}_{2(j-l+1)}\ga_{l}=0$.
  Indeed, we see from the definition of $\ga_l$ that
  \EQQS{
    M
    =\sum_{l=1}^{j}\IM\bar{b}_{2(j-l+1)}b_{2l}
      -\sum_{l=1}^{j}\sum_{k=1}^{l-1}\IM \bar{b}_{2(j-l+1)}\bar{a}_{2(l-k)}\ga_k
    =:A+B.
  }
  It is easy to see that $A=0$.
  We have
  \EQQS{
    B
    &=-\sum_{l=1}^{j-1}\sum_{k=1}^{j-l}
      \IM\bar{b}_{2(j-l-k+1)}\bar{a}_{2l}\ga_k\\
    &=-\sum_{l=1}^{j-1}(\RE a_{2l})\sum_{k=1}^{j-l}
      \IM\bar{b}_{2(j-l-k+1)}\ga_k
      +\sum_{l=1}^{j-1}(\IM a_{2l})\sum_{k=1}^{j-l}
        \RE\bar{b}_{2(j-l-k+1)}\ga_k=0
  }
  by \eqref{eq1.1} and the hypothesis.
  This shows that $\IM a_{2j+2}=0$ by the definiton of $\la_{2j+2}$.
  By Lemma \ref{rem_recur}, we obtain
  $\la_{2j+4}^-=\la_{2j+5}^-=0$, which completes the proof.
\end{proof}
\begin{rem}\label{rem_la}
  From the proof of the above lemma, we also see that
  $$\la_{2j^*+2}=2\IM a_{2j^*+2},\quad \la_{2j^*+4}=2\IM a_{2j^*+4}$$
  when $\la_{2j}=0$ for $1\le j\le j^*$.
\end{rem}

Now, we prove Proposition \ref{prop_energy_est_twist_+}.
\begin{proof}[Proof of Proposition \ref{prop_energy_est_twist_+}]
We give only the proof of \eqref{eq_twist_+} since we can show \eqref{eq_twist_-} in the same manner.
When $j^*=1$, we see from the definition that $\la_n^-=0$ for $n=1,2,3$.
When $j^*\ge2$, Lemma \ref{lem_la} implies that $\la_{j}^+=\la_j^-=0$ for $1\le j\le 2(j^*-1)$ and $\la_{2j^*-1}^-=\la_{2j^*}^-=\la_{2j^*+1}^-=0$.
This together with Remark \ref{rem2.1} implies $\la_{2j^*-1}^+\neq0$.
By \eqref{energy_ineq_+}, interpolation inequalities and the Young inequality, we have
  \EQQ{
    &\Biggr|\frac{d}{dt}\Biggr(\|P^+ u\|^2
    +\sum_{j=1}^{2m-1}\RE\al_j\LR{D_x^{-j} \overline{P^- u}, P^+ u}\Biggr)
    +\sum_{j=2j^*+2}^{2m-1} \la_j^-\||\ds|^{m-j/2}P^- u\|^2\\
    &+\la_{2j^*-1}^+\||\ds|^{m-j^*+1/2}P^+ u\|^2 \Biggr| \lesssim \|u\|^2+\||\ds|^{m-j^*}P^+ u\|^2
  }
Thus, we only need to show
  \EQ{\label{ee1}
    &\Biggr|\frac{d}{dt}\sum_{k=1}^{2(m-j^*-1)} \be_k^+ F_k^-(u)
    -\sum_{j=2j^*+2}^{2m-1} \la_j^-\||\ds|^{m-j/2}P^- u\|^2 \Biggr|\\
    &\lesssim \|u\|^2+\||\ds|^{m-j^*}P^+ u\|^2.
  }
Put $v=|\ds|^{-(k+2)/2} \zm u$.
Since $v$ satisfies \eqref{2mLSE}, by \eqref{energy_ineq_-}, we have
  \EQQ{
    &\Biggr|\frac{d}{dt}\Biggr(\| P^- v\|^2
    +\sum_{j=1}^{2m-1}\RE\al_j\LR{D_x^{-j} \overline{P^+ v}, P^- v}\Biggr)
    +\sum_{j=2j^*-1}^{2m-1}(-1)^j\la_j^+\||\ds|^{m-j/2}P^- v\|^2\Biggr| \\
    &\lesssim\|v\|^2+\| |\ds|^{m-j^*}P^+ v\|^2.
  }
Thus, we obtain
  \EQQ{
    &\Biggr|\sum_{k=1}^{2(m-j^*-1)} \be_k^+ \Big(\frac{d}{dt} F_k^-(u)
    +\sum_{j=2j^*-1}^{2m-k-3}(-1)^j\la_j^+\||\ds|^{m-(j+k+2)/2}P^- u\|^2\Big)\Biggr|\\
    & \lesssim \|u\|^2+\| |\ds|^{m-j^*}P^+ u\|^2.
  }
By \eqref{eq4.1}, we have
  \EQQS{
  &\sum_{k=1}^{2(m-j^*-1)} \sum_{j=2j^*-1}^{2m-k-3}(-1)^j\be_k^+\la_j^+\||\ds|^{m-(j+k+2)/2}P^- u\|^2\\
  &=\sum_{k=1}^{2(m-j^*-1)}\sum_{j=1}^{k}(-1)^{k-j+1}
    \be_j^+\la_{2j^*+k-j-1}^+\||\ds|^{m-j^*-(k+1)/2}P^- u\|^2.
  }
Therefore, by the definition of $\be_k^+$, we conclude \eqref{ee1}.
\end{proof}

%%%%%%%%%%%%%%%%%%%%%%%%%%%%%%%%%%%%%%%%%%%%%%%%%%%%%%%%%%%%%%%%%%%%%%%%%%%%%%%
%%%%%%%%%%%%%%%%%%%%%%%%%%%%%%%%%%%%%%%%%%%%%%%%%%%%%%%%%%%%%%%%%%%%%%%%%%%%%%%
%%%%%%%%%%%%%%%%%%%%%%%%%%%%%%%%%%%%%%%%%%%%%%%%%%%%%%%%%%%%%%%%%%%%%%%%%%%%%%%
%%%%%%%%%%%%%%%%%%%%%%%%%%%%%%%%%%%%%%%%%%%%%%%%%%%%%%%%%%%%%%%%%%%%%%%%%%%%%%%
%%%%%%%%%%%%%%%%%%%%%%%%%%%%%%%%%%%%%%%%%%%%%%%%%%%%%%%%%%%%%%%%%%%%%%%%%%%%%%%
%%%%%%%%%%%%%%%%%%%%%%%%%%%%%%%%%%%%%%%%%%%%%%%%%%%%%%%%%%%%%%%%%%%%%%%%%%%%%%%

\section{Proof of main theorem}
In this section, we show Theorem \ref{main_theorem}.
\begin{defn}
  For $f\in L^2(\M)$ and $N>0$, we define
  \EQQS{
    E(f;N)
    :=\|f\|^2+N\|\ds^{-m}\zm f\|^2
      +\sum_{j=1}^{m-1}\RE\ga_j\LR{D_x^{-2j}\zm\bar{f},\zm f}.
  }
  We choose $N$ sufficiently large so that Lemma \ref{lem_comparison} holds.
  If there is no confusion, we write $E(f):=E(f;N)$.
\end{defn}

\begin{lem}\label{lem_comparison}
  Let $N>0$ be sufficiently large.
  Then, for any $f\in L^2(\M)$ it holds that
  \EQQS{
    \frac{1}{2}E(f)\le \|f\|^2+N\|\ds^{-m}\zm f\|^2
    \le 2E(f).
  }
\end{lem}

\begin{proof}
  The Gagliardo-Nirenberg inequality and the Young inequality show that
  \EQS{\label{eq3.1}
    \sum_{j=1}^{m-1}|\RE\ga_j\LR{D_x^{-2j}\zm\bar{f},\zm f}|
    \le \frac{1}{2}\|f\|^2+C\|\ds^{-m}\zm f\|^2.
  }
  So, it suffices to choose $N=2C$.
\end{proof}

We prove the first part of Theorem \ref{main_theorem}.

\begin{proof}[Proof of ``Dispersive type'' in Theorem \ref{main_theorem}]
  We consider our problem only on $[0,\I)$ since the result on $(-\infty,0]$ follows from the same argument.
  Let $T>0$, which can be arbitrary large.
  We first show the a priori estimate $\sup_{t\in[0,T]}\|u(t)\|\le C\|\vp\|$.
  We assume that $u$ satisfies \eqref{2mLSE} and \eqref{initial}.
  Then, it is easy to see that
  $\frac{d}{dt}\|\ds^{-m}\zm u\|^2\le 2|\LR{D_t \ds^{-2m}\zm u,\zm u}| \le C\|u\|^2$.
  This together with \eqref{energy_ineq_0}, Lemma \ref{lem_comparison} and
  $\la_j=0$ for $1\le j\le 2m-1$ implies that
  $\frac{d}{dt} E(u(t))\le CE(u(t))$ on $[0,T]$.
  Thus, by the Gronwall inequality and Lemma \ref{lem_comparison}, we obtain the a priori estimate.
  Next, we show the existence.
  Let $\vp_n=\F^{-1} \chi(|\x|<n) \F \vp$ for $n\in \N$.
  Then, we have the solution $u_n$ of \eqref{2mLSE} with $u_n(0)=\vp_n$ by \eqref{explicit_formula}.
  Moreover, $u_n \in C([0,T];L^2(\M))$ since $|\sum_{j=0}^{2m}\x^{2m-j}X_j| \le C(\{a_j\},\{b_j\},n)$ for $|\x|<n$.
  Since $\{\vp_n\}$ is a Cauchy sequence in $L^2(\M)$, by the a priori estimate, we conclude $\{u_n\}$ is
  also a Cauchy sequence in $C([0,T];L^2(\M))$. Thus, we obtain the solution $u \in C([0,T];L^2(\M))$
  of \eqref{2mLSE}--\eqref{initial} as the limit of $u_n$.
  Finally, the uniqueness easily follows from the a priori estimate.
\end{proof}

\begin{proof}[Proof of ``Parabolic type'' in Theorem \ref{main_theorem}]
  We use the argument from the proof of Theorem 1.2 in \cite{Tsugawa17}.
  We consider only the case $\la_{2j^*}>0$ since the other case follows from the same argument.
  Let $T>0$, which can be arbitrary large.
  By the Gagliardo-Nirenberg inequality and the Young inequality, we have
  \EQQ{
    \biggl|\sum_{j=j_*+1}^{m-1}\la_{2j}\||\ds|^{m-j} u\|^2
     +\sum_{j=j_*+1}^{m}\la_{2j-1}\LR{D_x^{2(m-j)+1} u, u}\biggr|
     \le \frac{1}{2}\la_{2j^*}\||\ds|^{m-j^*}u\|^2+C\|u\|^2.
  }
  Recall that $\la_j=0$ for $1\le j\le 2j^*-1$.
  Therefore, in the same manner as the proof of ``Dispersive type,'' we obtain the a priori estimate:
  \EQQS{
    \sup_{t\in[0,T]}
    \Biggr(\|u(t)\|^2+\frac{\la_{2j^*}}{2}\int_0^t\||\ds|^{m-j^*}u(\ta)\|^2d\ta\Biggr)\le C\|\vp\|^2.
  }
  It then follows that we have the unique existence of the solution
  $u \in C([0,T];L^2(\M)) \cap L^2([0,T];H^{m-j^*}(\M))$, which implies that
  $u(t) \in H^{m-j^*}(\M)$ for a.e. $t\in[0,T]$.
  Let $0<\e<T$.
  Then there exists $t_0\in(0,\e/2)$ such that $u(t_0)\in H^{m-j^*}(\M)$.
  Since $\LR{\p_x}^{m-j^*} u$ satisfies \eqref{2mLSE}--\eqref{initial} with initial data
  $\vp:=\LR{\p_x}^{m-j^*}u(t_0)\in L^2(\M)$,
  applying the same argument as above, we conclude
  $\LR{\p_x}^{m-j^*}u \in C([t_0,T];L^2(\M)) \cap L^2([t_0,T];H^{m-j^*}(\M))$.
  That is, $u\in C([t_0,T];H^{m-j^*}(\M))\cap L^2([t_0,T];H^{2(m-j^*)}(\M))$.
  We can choose $t_1$ so that $\e/2<t_1<\e/2+\e/4$ and $u(t_1)\in H^{2(m-j^*)}(\M)$.
  Again, applying the same argument as above with the initial data $\vp:=\LR{\p_x}^{2(m-j^*)}u(t_0)\in L^2(\M)$,
  we conclude $u\in C([t_1,T];H^{2(m-j^*)}(\M))\cap L^2([t_1,T];H^{3(m-j^*)}(\M))$.
  By repeating this process, we conclude $u\in C([\e,T];H^{k(m-j^*)}(\M))$ for any $k \in \N$, which implies
  $u\in C^\ell([\e,T];H^{k(m-j^*)-2m\ell}(\M))$ for any $k, \ell \in \N$ by \eqref{2mLSE}.
  By the Sobolev embedding, we obtain $u\in C^\infty([\e,T]\times\M)$.
  Since we can take $\e>0$ arbitrary small and $T>0$ arbitrary large, we conclude
  $u\in C^\I((0,\infty)\times\M)$.
  Finally, we show the nonexistence result by contradiction.
  Assume that there exists a solution $u \in C((-\de,0]; L^2(\M))$ of \eqref{2mLSE}--\eqref{initial}
  with $\vp \in L^2(\M)\setminus C^\infty(\M)$.
  We take $t_0$ such that $-\de <t_0<0$.
  Then, as we proved above, we have $u\in C^\I((t_0,0]\times\M)$, which contradicts
  to the assumption $\vp=u(0) \not\in C^\infty(\M)$.
\end{proof}

The following proposition is the main tool to show the result for ``Elliptic type'' in Theorem \ref{main_theorem}.
\begin{prop}[A smoothing for ``Elliptic type"]\label{prop_smoothing}
Let $u\in C([t_0,t_1]; L^2(\M))$ satisfy \eqref{2mLSE}.
Assume that there exists $j^*\in \N$ such that $\la_j=0$ for $1\le j < 2j^*-1$ and $\la_{2j^*-1}>0$ (resp. $<0$).
Then, it follows that
\EQS{
&P^+ u \ (\text{resp.} \, P^- u)\in C((t_0,t_1];H^{1/2}(\M)) \quad(\text{forward smoothing}),\label{eq_smoothing1}\\
&P^- u \ (\text{resp.} \, P^+ u)\in C([t_0,t_1);H^{1/2}(\M)) \quad(\text{backward smoothing}).\label{eq_smoothing2}
}
In particular, it holds that $u\in C^\infty((t_0,t_1)\times \M)$.
\end{prop}

\begin{proof}
  We consider only the case $\la_{2j^*-1}^+>0$ since the same proof works for the case $\la_{2j^*-1}^+<0$.
  For simplicity, set
  $$G^+(u):=\sum_{j=1}^{2m-1}\RE\al_j\LR{D_x^{-j} \overline{P^- u}, P^+ u}+\sum_{k=1}^{2(m-j^*-1)} \be_k^+ F_k^-(u),$$
  where $F_k^-$ is defined in Proposition \ref{prop_energy_est_twist_+} and $\{\al_j\}$ and $\{\be_k\}$
  are defined in Definiton \ref{def2}.
  Set $M:=\sup_{t\in[t_0,t_1]}\|u(t)\|$.
  Note that $\sup_{t\in[t_0,t_1]} (|G^+(u(t))| + |\, G^+(|\p_x|^{1/2}u(t))|) \le CM$
  and $G^+(|\p_x|^{1/2}u(t))$ is continuous on $[t_0,t_1]$
  by the presence of $D_x^{-j}$ in the definition of $G^+(u)$ above.
  By the Gagliardo-Nirenberg inequality and the Young inequality, we have
  \EQQ{
    \| |\ds|^{m-j^*}Qu\|^2\le \de \| |\ds|^{m-j^*+1/2}Qu\|^2+C\de^{-1}\|u\|^2
  }
  for $\de>0$, $Q=P^+$ or $P^-$.
  Take $\de>0$ sufficiently small. Then, this together with \eqref{eq_twist_+} and \eqref{eq_twist_-} yields
  \EQQS{
    \la_{2j^*-1}^+\int_{t_0}^{t_1}\||\ds|^{m-j^*+1/2}Q u(\ta)\|^2d\ta
    \le C(M)(1+|t_1-t_0|),
  }
 for $Q=P^+$ or $P^-$.
  By the interpolation, we also have
  \EQ{\label{e2}
  &\int_{t_0}^{t_1}\||\ds|^{s}u(\ta)\|^2d\ta\\
  &=\int_{t_0}^{t_1}(\||\ds|^{s}P^-u(\ta)\|^2+\||\ds|^{s}P_0u(\ta)\|^2+\||\ds|^{s}P^+u(\ta)\|^2)d\ta\\
  &\le C(M,\la_{2j^*-1}^+)(1+|t_1-t_0|)
  }
  for $0\le s\le m+j^*-1/2$.
  It then follows that $\||\ds|^{m-j^*+1/2} u(t)\|<\I$ for a.e. $t\in[t_0,t_1]$.
  Then, for any $\e>0$ there exists $t_*\in(t_0,t_0+\e)$ such that $\||\ds|^{m-j^*+1/2} u(t_*)\|<\I$.
  Note that \eqref{eq_twist_+} holds even if we replace $u$ with $|\ds|^{1/2}u$ since $|\ds|^{1/2}u$ satisfies \eqref{2mLSE}.
  Thus,
  \EQ{\label{e1}
    &\bigg|\frac{d}{dt}\big(\||\ds|^{1/2}P^+ u\|^2
    +G^+(|\ds|^{1/2}u) \big) + \la_{2j^*-1}^+\||\ds|^{m-j^*+1}P^+ u\|^2\bigg|\\
    &\le C \||\ds|^{1/2}u\|^2+ C\||\ds|^{m-j^*+1/2}P^+ u\|^2,
  }
  By the Gagliardo-Nirenberg inequality and the Young inequality, we have
  \EQQ{
    \| |\ds|^{m-j^*+1/2}P^+ u\|^2\le \de \| |\ds|^{m-j^*+1}P^+u\|^2+C\de^{-1}\||\ds|^{1/2}u\|^2
  }
  for $\de>0$.
  Taking $\de>0$ sufficiently small and integrating \eqref{e1} on $[t_*,t](\subset[t_0,t_1])$ with \eqref{e2}, we obtain
  \EQ{\label{e3}
    &\||\ds|^{1/2}P^+ u(t)\|^2
     +\frac{\la_{2j^*-1}^+}{2}\int_{t_*}^t\||\ds|^{m-j^*+1}P^+u(\ta)\|^2d\ta\\
    &\le C(M,\la_{2j^*-1}^+,|t_1-t_0|)+\||\ds|^{1/2}P^+ u(t_*)\|^2<\I
  }
  since $u(t_*)\in H^{m-j^*+1/2}(\M)$.
  Therefore, by \eqref{e1} again, it follows that
  for any $t_*\le t'\le t\le t_1$
  \EQQS{
    &\Big|\||\ds|^{1/2} P^+ u(t)\|^2-\||\ds|^{1/2} P^+ u(t')\|^2\Big|\\
    &\le \Big|\Big[|\ds|^{1/2} P^+ u(\ta)\|^2
     +G^+(|\ds|^{1/2}u)\Big]_{\ta=t'}^{\ta=t}
     +\la_{2j^*-1}^+\int_{t'}^t\||\ds|^{m-j^*+1}P^+u(\ta)\|^2d\ta\Big|\\
    &\quad+\Big|\Big[G^+(|\ds|^{1/2}u)\Big]_{\ta=t'}^{\ta=t}\Big|
     +\la_{2j^*-1}^+\int_{t'}^t\||\ds|^{m-j^*+1}P^+u(\ta)\|^2d\ta\\
    &\le C\int_{t'}^{t}\||\ds|^{1/2}u(\ta)\|^2d\ta
     +C\int_{t'}^t\||\ds|^{m-j^*+1/2}P^+u(\ta)\|^2d\ta\\
    &\quad+\la_{2j^*-1}^+\int_{t'}^t\||\ds|^{m-j^*+1}P^+u(\ta)\|^2d\ta
    +\Big|\Big[G^+(|\ds|^{1/2}u)\Big]_{\ta=t'}^{\ta=t}\Big|.
  }
  \eqref{e2}, \eqref{e3} and the dominated convergence theorem imply that the right-hand side goes to $0$ as $|t-t'|\to 0$, which shows that $\||\ds|^{1/2} P^+ u(t)\|$ is continuous on $[t_*,t_1]$.
  %By \eqref{explicit_formula}, $u_n:=\F^{-1} \chi(|\x|<n) \F u \in C([t_0,t_1];H^{1/2}(\M))$ for $n\in \N$.
  The fact $P^+u\in C([t_0,t_1];L^2(\M))$ with $P^+u\in L^\infty([t_*,t_1];H^{1/2}(\M))$ yields $P^+u \in C_w([t_*,t_1];H^{1/2}(\M))$.
  Combining the continuity of $\||\ds|^{1/2} P^+ u(t)\|$ and the weak continuity of $P^+ u(t)$ in $H^{1/2}(\M)$,
  we obtain $P^+u\in C([t_*,t_1];H^{1/2}(\M))$.
  Since we can take $\e>0$ arbitrary small, we get
  $P^+u\in C((t_0,t_1];H^{1/2}(\M))$.
  We also obtain $P^-u\in C([t_0,t_1);H^{1/2}(\M))$ in the same manner.
  Therefore, $u=P^{-} u+P_0 u+P^{+} u \in C((t_0,t_1);H^{1/2}(\M))$.
  By repeating this process, we also obtain $u \in C((t_0,t_1);H^{k/2}(\M))$ for any $k\in \N$,
  which yields $u \in C^\infty((t_0,t_1)\times \M))$ since $u$ satisfies \eqref{2mLSE}.
\end{proof}

\begin{proof}[Proof of ``Elliptic type'' in Theorem \ref{main_theorem}]
  We use the argument from the proof of Theorem 1.2 in \cite{TsugawaP}.
  We consider only the case $\la_{2j^*-1}>0$ since the case $\la_{2j^*-1}<0$ follows from the same argument.
  Let $\vp\in L^2(\M)$ satisfy $P^+\vp\notin H^{1/2}(\M)$.
  We prove Theorem \ref{main_theorem} by contradiction.
  We assume that there exists $u \in C([-\de,0];L^2(\M))$ satisfying \eqref{2mLSE}--\eqref{initial} on $[-\de,0]$.
  Then, we have $P^+u\in C((-\de,0];H^{1/2}(\M))$ by Proposition \ref{prop_smoothing}.
  However, it contradicts to $P^+\vp\notin H^{1/2}(\M)$.
  This proof works even if we replace $P^+$ and $[-\de,0]$ with $P^-$ and $[0,\de]$, respectively.
  Similarly, we can show that for any $\de>0$ there exist no solution $u(t,x)$ of \eqref{2mLSE}--\eqref{initial} with $u(0,x)=\vp(x)\in L^2(\M)\setminus C^\I(\M)$ on $[-\de,\de]$ satisfying $u\in C([-\de,\de];L^2(\M))$.
\end{proof}

\section*{Acknowledgements}
The first author was supported by RIKEN Junior Research Associate Program and JSPS KAKENHI Grant Number JP20J12750.
The second author was supported by JSPS KAKENHI Grant Number 17K05316.

\end{document}